\documentclass{birkjour}

\usepackage{amssymb}
\usepackage{amsmath}
\usepackage{amsthm}
\usepackage{shortcuts}

\begin{document}

\title{The normal map as a vector field}

\author{Thomas Waters$^*$ \and Matthew Cherrie}

\address{School of Mathematics and Physics, University of Portsmouth, England PO13HF}

\email{thomas.waters@port.ac.uk}

\begin{abstract}
In this paper we consider the normal map of a closed plane curve as a vector field on the cylinder. We interpret the critical points geometrically and study their Poincar\'{e} index, including the points at infinity. After projecting the vector field to the sphere we prove some counting theorems regarding the winding and rotation index of the curve and its evolute. We finish with a description of the extension to focal sets of surfaces.
\end{abstract}

\maketitle

\section{Introduction}

The envelope/caustic/focal set of the normal lines to a plane curve is known as the `evolute' and is a classical topic in Differential Geometry (see Figure 1 for an example). Evolutes have been studied by many authors over the years: most well known texts follow a traditional approach (\cite{docarmo},\cite{bergergostiaux},\cite{kuhnel}), some treat the topic in terms of contact and singularity (\cite{curvessing},\cite{porteous},\cite{izumiya}) and others take a more abstract approach (\cite{arnold}, \cite{tabachfuchs}, \cite{guilf}, \cite{evo1}). While known and studied as far back as Huygens \cite{yoder} (or even earlier \cite{appo}), one reason for interest in evolutes is as a Euclidean analogue for the conjugate locus (\cite{poincare}, \cite{myers}, \cite{TWclocconv}, \cite{hannes}); indeed Jacobi's original formulation of his last geometric statement, recently proved by Itoh and Kiyohara \cite{Itoh1}, was that the conjugate locus of a generic point on the triaxial ellipsoid has ``die gestalt der evolute der ellipse''\cite{jacobi}.

\begin{figure}
\begin{center}
{\includegraphics[width=0.48\textwidth]{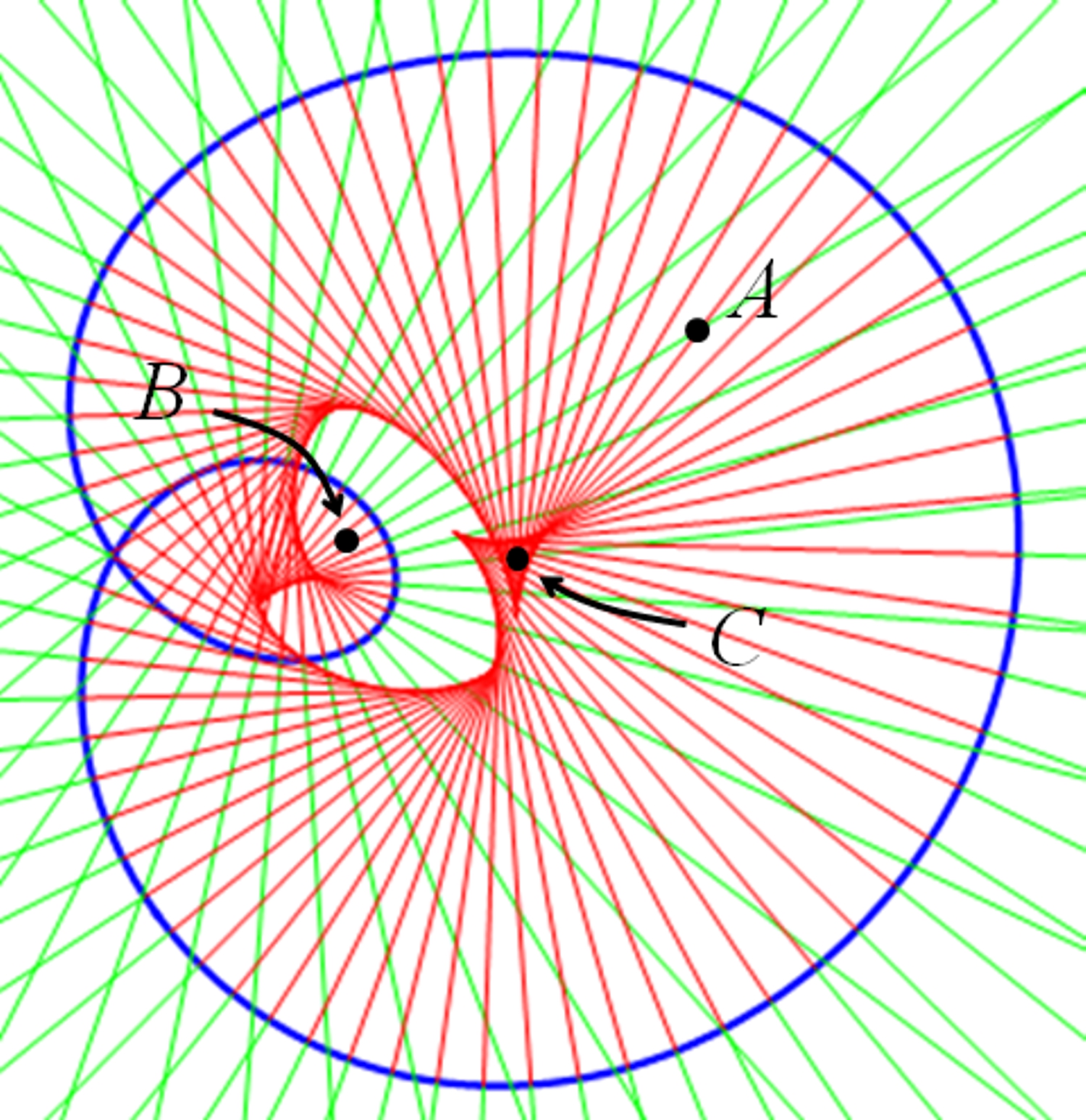}\hspace{0.25cm}\includegraphics[width=0.47\textwidth]{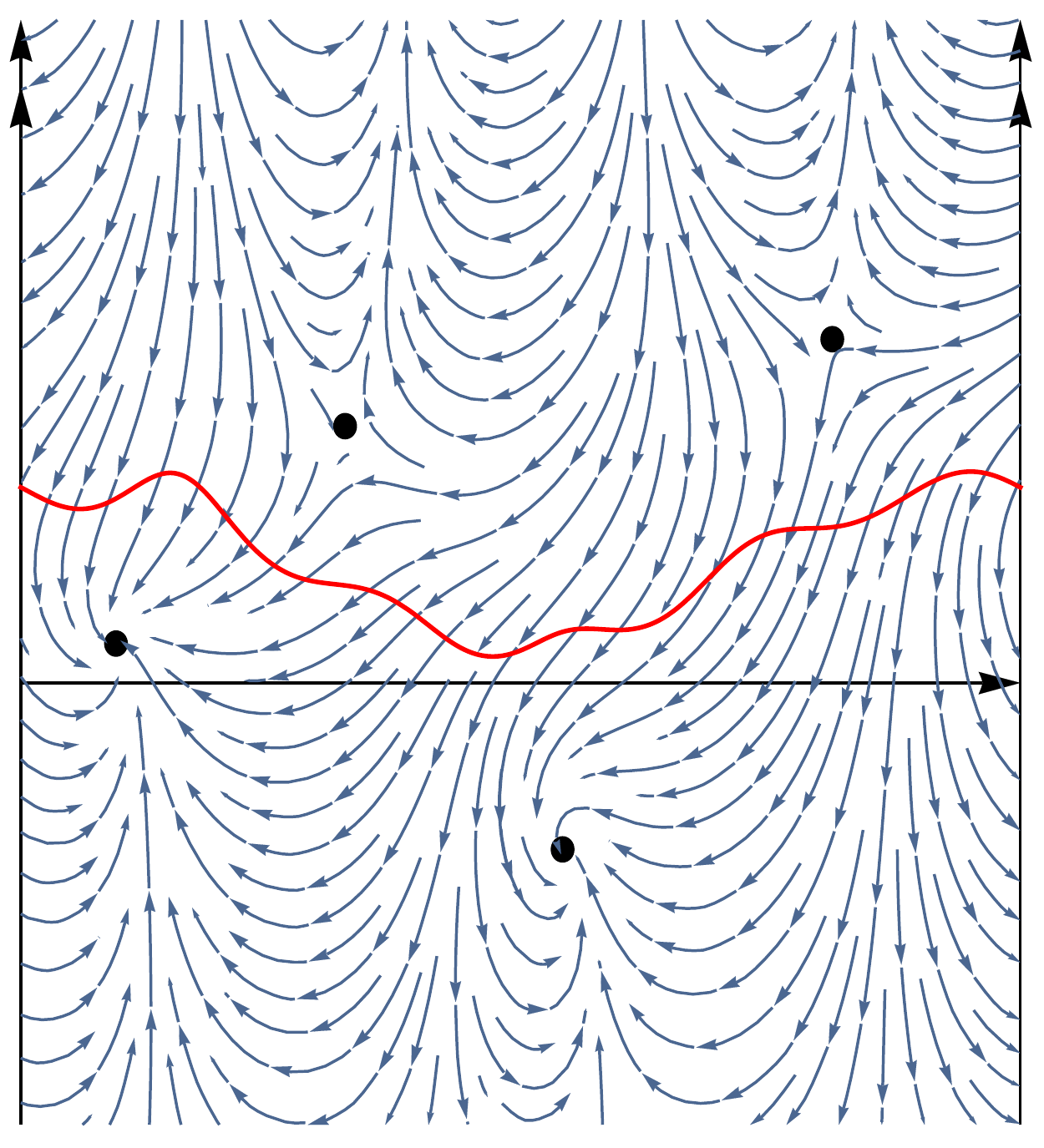}}\caption{On the left in blue is a plane curve of the type described at the beginning of Section 2, in red its evolute and the normal segments with $\rho<0$ and $0<\rho<1/k$ in green and red repsectively. On the right the normal map of the same curve (with origin at $A$) as a vector field; in red the curve $\rho=1/k$ and the dots are critical points (axes are $s,\rho$).}\label{normspic}
\end{center}
\end{figure}

This paper will take a novel approach: we view the normal map of a closed plane curve as a vector field, whose phase space is the cylinder. Critical points of this vector field have a Poincar\'{e} index with a clear geometrical meaning, and the index at infinity can be found in terms of the rotation index of the original curve. Then by projecting to the 2-sphere we prove certain counting results concerning the winding index of the plane curve and its evolute. Before we begin we will fix some terminology and notation:

The degree of a map is a central notion in Differential Topology \cite{degree}, and we are primarily interested in the degree of a map from $\rS^1$ to $\rS^1$. There are two ways to view this degree: the sum of orientation-defined signs attached to the preimages of a regular point (see for example \cite{milnor}), or the total change in the lifted angle made by the map in the sense of covering spaces (see for example \cite{docarmo} or \cite{lovett}); but intuitively if $f:\rS^1\to\rS^1$ then $deg(f)$ is the number of times the codomain is covered as we traverse the domain once, taking orientation into account. There are three main examples, all of which we will use in this work:

\begin{enumerate}
\item Let $\ba:\rS^1\to\rree$ be a closed, regular, smooth, positively oriented plane curve. Then the {\it winding index} of $\ba$ about the origin, denoted $w^\alpha_o$,  is $deg(\ba/|\ba|)$.
\item With $\ba$ as defined above, if $\dot{\ba}$ denotes the derivative w.r.t.\ some appropriate
parameter, then the {\it rotation index} of $\ba$, denoted $r^\alpha$, is $deg(\dot{\ba}/|\dot{\ba}|)$. 
\item If $\bv:\rree\to\rree$ is a smooth vector field with an isolated critical point at $p$, i.e.\ $\bv(p)=\bmt{0}$, let $\bv^*$ denote $\bv$ restricted to a smooth simple positively oriented closed curve that contains $p$ and no other critical point of $\bv$. Then the {\it Poincar\'{e} index} of $p$, denoted $i_p$, is $deg(\bv^*/|\bv^*|)$.
\end{enumerate} All of these definitions have various extensions which we may need to make use of in the text, in particular it is customary to consider the Poincar\'{e} index of a simple closed positively oriented curve in the domain of a vector field, which is equal to the sum of the Poincar\'{e} indices of the critical points contained by the curve \cite{jands}.

In the next section we will introduce the normal map of a plane curve with positive curvature, its connection with the evolute and its representation as a vector field on the cylinder. We describe the meaning and structure of the critical points of this vector field, and how they respond to moving the curve around in the plane. In Section 3 we project the vector field onto a sphere and prove certain counting theorems concerning the rotation and winding index of the plane curve and its evolute, in terms of quantities which are easily read off the vector field picture. In Section 4 we extend the discussion to the normal map of a surface in $\rre^3$.

\section{Finite critical points}

Suppose $\ba$ is a closed, regular, smooth plane curve, with $s$ as arc-length parameter. The exposition of what follows is clearest is we assume the curvature $k$ has one sign; we therefore orient $\ba$ such that $k>0$ if $(\bt,\bn)$ (the unit tangent and normal respectively) are right-handed. At any point of $\ba$, we may travel  a distance $\rho$ in the direction of $\bn$, to arrive at the point in the plane $\ba+\rho\bn$. This map, \[ F(s,\rho):\rS^1\times\rre\to\rre^2:(s,\rho)\to \ba(s)+\rho\bn(s) \] is known as the normal map. As well as providing a coordinate system adapted to the curve, the normal map is of interest since its singular set (the set of points where the map is not invertible) defines the evolute of $\ba$ (or rather its image under $F$ does). To see this we find the Jacobian of $F$: \begin{align} DF=\begin{vmatrix}
\dot{\ba}+\rho\dot{\bn} \\ \bn
\end{vmatrix}=\begin{vmatrix}
\bt+\rho(-k\bt) \\ \bn
\end{vmatrix}=1-\rho k \label{df}\end{align} and therefore the image of the singular set of this map is the curve $\bbet(s)=\ba+(1/k)\bn$; since $k>0$ then $\bbet$ is closed and bounded. See Figure \ref{normspic} on the left for an example.

The main idea of this paper is to view the normal map as a vector field whose domain is $\rS^1\times\rre$, the cylinder. An example is in Figure \ref{normspic} on the right. The first thing to notice is the existence of critical points (those with $\rho<\infty$ we will call `finite critical points'); geometrically a critical point at $(s^*,\rho^*)$ means the normal line to $\ba$ at $s^*$ passes through the origin, after travelling a distance $\rho^*$ in the direction of the positively oriented unit normal. The nature of the critical points can be classified as follows:

\begin{thm}
Suppose $\ba$ is in generic position, meaning its evolute does not pass through the origin. Then the finite critical points have Poincar\'{e} index $+1$ (which we shall call `centres') or $-1$ (which we shall call `saddles'); the saddles all lie above the line $\rho=1/k$, and the centres lie below it. There are as many saddles as there are centres (so the sum of the Poincar\'{e} indices of the finite critical points is zero), and there is at least one of each. 
\end{thm}
\begin{proof}
If the evolute does not pass through the origin then the critical points are non-degenerate and their Poincar\'{e} index is, from \eqref{df}, the sign of $1-\rho k$ (see \cite{milnor}); therefore saddles have $\rho>1/k$ and centres $\rho<1/k$. To distinguish saddles and centres geometrically we recall the distance squared function $D^2:\rS^1\to\rre^+:s\to\ba.\ba$ which is stationary when $\ba.\dot{\ba}=0$, i.e. when the position vector at $s=s^*$ is perpendicular to the tangent line. Hence the normal to $\ba$ at $s^*$ passes through the origin, in fact $\ba(s^*)=-\rho^*\bn(s^*)$. Thus critical points of the vector field correspond to stationary points of the $D^2$ function. To classify the stationary points we note \[ \frac{d^2 (D^2)}{ds^2}=\dot{\ba}.\dot{\ba}+\ba.\ddot{\ba}=1-\rho k.\] Therefore critical points with Poincar\'{e} index $-1$ (saddles) correspond to maxima of $D^2$, and those with Poincar\'{e} index $+1$ (centres) correspond to minima. Since $D^2$ is defined on a circle there must be at least one maximum and one minimum,  and since $\chi(\rS^1)=0$ there must be as many saddles as centres.
\end{proof}

An immediate advantage to the vector field approach is that we can tell at a glance from the right of Figure \ref{normspic} how many normals to $\ba$ pass through the origin; this is a locally constant function on $\rre^2/\bbet$. We can also quickly see what point on $\ba$ is closest to the origin, and how far it is. Suppose we were interested in some point other than the origin, say $p$? We can either move our coordinate system or move the curve so the origin is at $p$. As we do, we may find the origin crosses the evolute and then the number of normals through the origin will change by $\pm 2$; this is manifested in the vector field by the creation or annihilation of a saddle-centre pair of critical points. 

To sketch the details we suppose there is a critical point on the curve $\rho=1/k$, and then we perturb it. The exposition is clearest if we suppose we translate $\ba$ in such as way that the origin passes though the evolute perpendicularly, i.e.\ we let $\ba\to\ba+\epsilon\bt$. Then, to second order in $\delta s,\delta \rho$, the normal map looks like \[ F(s^*+\delta s,\rho^*+\delta\rho)=\bn[\delta\rho+O(3)]+\bt[\epsilon-\tfrac{1}{2}\rho\dot{k}\delta s^2-k\delta s\delta\rho+O(3)] \] (we have dropped the $^*$ on the RHS). Where are the critical points? Since $\bt$ and $\bn$ are orthonormal we require $\delta\rho=0$ which means $\delta s=\pm\sqrt{2\epsilon/\rho\dot{k}}$. Therefore if $2\epsilon/\rho\dot{k}>0$ there will be two critical points, a saddle and a centre on either side of the curve $\rho=1/k$; otherwise none.

The following corollary is now as easy to prove as it is to state: suppose $p$ is some point in the plane not on $\ba$ or $\bbet$; emanating from $p$ are an even number of line segments that reach $\ba$ normally; precisely half of these line segments are tangential to $\bbet$ on the way to meeting $\ba$, the other half reach $\ba$ first. As proof we note that, after shifting the origin to $p$, there are as many saddles as there are centres and the saddles all have $\rho>1/k$, which means the normals to $\ba$ pass through their focal point before reaching $p$; the centres have $\rho<1/k$ and so the normals reach $p$ first.

\section{Global results}

To understand the global properties of the vector field we need to look at infinity. The cylinder has two points at infinity, both $\rho\to +\infty$ and $\rho\to -\infty$. A two-point compactification of the cylinder could be the Mercator projection: we orient the cylinder in $\rre^3$ so $\bsh\times\brh$ is outward pointing, and consider the sphere tangential to the cylinder at $\rho=0$. A point on the cylinder $Q$ is mapped to the point on the sphere $Q'$ which is the intersection of the line segment connecting $Q$ to $O$ with the sphere. In what follows it may also be necessary to further vertically project a hemisphere of the sphere onto its equatorial plane. To visualize the vector field on the sphere we may find the integral curves on the cylinder and plot their projections on the sphere; an example is in Figure \ref{icsphpic}. We can see clearly the saddles and centres as before (the Poincar\'{e} index is preserved under the projection), but the points of interest are now the north and south poles, $N$ and $S$, as these correspond to the points at infinity on the cylinder. Clearly these points are not saddles or centres; we will derive their Poincar\'{e} index shortly. First we prove the following lemma as it will be useful to make rigorous the rest of this section.

\begin{figure}
\begin{center}
\includegraphics[width=0.8\textwidth]{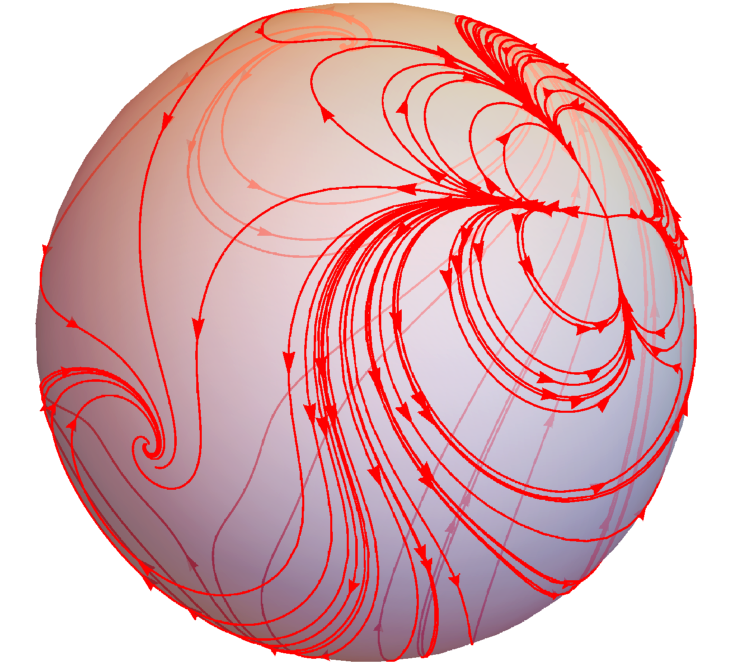}\caption{The normal map vector field from Figure \ref{normspic} projected onto the unit sphere. The north and south poles have Poincar\'{e} index 3 and -1 respectively, as predicted in Theorem 2.}\label{icsphpic}
\end{center}
\end{figure}

\begin{lem}
Let $\bu,\bv,\bw$ be unit tangent vector fields on the cylinder, and let $\bg$ be a simple closed positively oriented but not contractible curve on the $\rho>0$ portion of the cylinder. We will denote by $deg(\bu,\bv;\bg)$ the total anticlockwise rotation of $\bu$ with respect to $\bv$ as we perform one circuit of $\bg$. Let $(\phantom{a})'$ denote an object after projection from the cylinder to the sphere and then to the equatorial plane. Then \begin{enumerate}
\item $deg(\bu,\bv;\bg)=deg(\bu,\bw;\bg)+deg(\bw,\bv;\bg)$,
\item $deg(\bu',\bv';\bg')=deg(\bu,\bv;\bg)$.
\end{enumerate}
\end{lem}
\begin{proof}
The first statement is clear but worth stating for what follows. For the second statement, we note that the projections described above do not preserve angle but they do preserve parallelism and orientation. Therefore $\bu'=\bv'$ if and only if $\bu=\bv$ and $\bu'$ will cross $\bv'$ in the same sense that $\bu$ crosses $\bv$. Therefore in traversing $\bg'$ we see $\bu'$ will have completed as many anticlockwise rotations about $\bv'$ as $\bu$ will complete about $\bv$ as we traverse $\bg$.
\end{proof}

\begin{thm}
If $\ba$ is a closed plane curve as described in Section 2 with rotation index $r^\alpha$, then the normal map associated with $\ba$ as a vector field projected onto the sphere has Poincar\'{e} indices \begin{align} i_N=1+r^\alpha,\qquad i_S=1-r^\alpha. \label{iN} \end{align}
\end{thm}

\begin{proof}
Let $\bg$ be the curve $\rho=\rho_0$ on the cylinder, and let $\bu$ be the normalized $F(s,\rho_0)$. As $\rho\to+\infty$, $\bg'$ is a small positively oriented circle enclosing $O$, the vertical projection of $N$ onto the equatorial plane. If $\bk$ is some fixed direction in this plane, then $i_N=deg(\bu',\bk;\bg')$ but by the previous lemma \begin{align*} deg(\bu',\bk;\bg')&=deg(\bu',\bsh';\bg')+deg(\bsh',\bk;\bg') \\ &=deg(\bu,\bsh;\bg)+1. \end{align*} Since \[ \lim_{\rho\to\infty}\bu=\lim_{\rho\to\infty}\frac{F(s,\rho_0)}{|F(s,\rho_0)|}=\bn \] then $deg(\bu,\bsh;\bg)\to deg(\bn,\bsh;\bg)=r^\alpha$ and hence $i_N=1+r^\alpha$. For $i_S$ we can repeat the construction for $\rho\to -\infty$, or note that the sum of the Poincar\'{e} indices of all the critical points on the sphere must sum to 2, whereas from Theorem 1 the sum of the Poincar\'{e} indices at the finite critical points is zero.
\end{proof}

We should note for the next section that there are two other ways of proving this theorem: if $\ba$ does indeed have rotation index $r^\alpha$, then $\bn$ is regular homotopic to $(\cos(2\pi r^\alpha s/l),\sin(2\pi r^\alpha s/l))$ where $l$ is the length of $\ba$. Now we can either explicitly project this vector field onto the sphere and then the plane and calculate $i_N$ directly, or we can use the methods of Firby and Gardiner \cite{firby} and note that each anticlockwise rotation of $\bn$ on the cylinder generates an elliptic sector in the plane. We have instead used the approach of the lemma as it leads nicely to the following theorems. First we define a `normal segment' to be the normal to $\ba(s)$ with $-\infty<\rho<1/k(s)$.

\begin{thm}
Let $\ba$ be a closed plane curve as before, and $\bbet$ be its evolute. Let $n_p$ be the number of normal segments to $\ba$ through a point in the plane $p$ (assuming $p$ is not on $\ba$ or $\bbet$). Then \begin{align} w^\beta_p=r^\alpha -n_p. \label{wbp} \end{align}
\end{thm}
\begin{proof} A translation of coordinates puts the origin at $p$. Let $\bg$ be the curve $\rho=1/k$ on the cylinder, and let $\bu$ be the normal map restricted to this curve and normalized. Then $deg(\bu,\bsh;\bg)=w^\beta_p$, but also if $\bv$ is the tangent vector to $\bg$ then \[ w^\beta_p=deg(\bu,\bsh;\bg)=deg(\bu,\bv;\bg)+deg(\bv,\bsh;\bg)=deg(\bu,\bv;\bg). \] Now projecting to the sphere and hence to the plane, $\bg'$ is a simple closed positively oriented  curve which contains $n_p$ saddle points and the projection of $N$; therefore from Theorems 1 and 2 its Poincar\'{e} index is \[ n_p(-1)+(1+r^\alpha). \] But this is the total rotation of $\bu'$ with respect to $\bk$ as we traverse $\bg'$, and \begin{align*} deg(\bu',\bk;\bg')&=deg(\bu',\bv';\bg')+deg(\bv',\bk;\bg') \\ &=deg(\bu,\bv;\bg)+1 \end{align*} which means \[ n_p(-1)+(1+r^\alpha)=w^\beta_p+1 \] and the result follows.
\end{proof}

We can rephrase the last result in terms of simply the number of normals to $\ba$ through $p$. If we denote this by $N_p$ then (since there are as many saddles as there are centres) $N_p=2n_p$ and \begin{align} w^\beta_p=r^\alpha -N_p/2, \label{wpbN} \end{align} which extends the result of \cite{khim} to non-simple curves.

Now suppose we instead consider separately the portions of the normal segments with $\rho<0$ and $0<\rho<1/k$. Then we have:

\begin{thm}
With the definitions as before, \begin{align} w^\alpha_p=r^\alpha-m_p \label{wap} \end{align} where $m_p$ is the number of normal segments with $\rho<0$ through $p$.
\end{thm}
\begin{proof}
As in the previous theorem, only now we let $\bg$ be the curve $\rho=0$ on the cylinder. There will be $m_p$ centres below this line, and if $\bu$ is $F(s,0)$ normalized then $deg(\bu,\bsh;\bg)=w^\alpha_p$.
\end{proof}

Looking at Figure \ref{normspic} left and right we see examples of both the previous theorems in action: the point $A$ has two normal segments through it, the red line corresponding to the critical point above the horizontal axis but below the line $\rho=1/k$, and the green line the critical point below the axis; therefore $n_A=2$ and $m_A=1$ and we see $w^\beta_A=2-2=0$ and $w^\alpha_A=2-1=1$. For the point $B$ we observe $n_B=1,m_B=0$ and $w^\beta_B=2-1=1$ and $w^\alpha_B=2-0=2$. Finally we observe $n_C=3,m_C=1$ and $w^\beta_C=2-3=-1$ and $w^\alpha_C=2-1=1$.

We can see now the advantage of viewing the normal map as a vector field: we can see at a glance the number of centres $n_p$, the number of centres below $\rho=0$ (this is $m_p$), and the index at the north pole $i_N$. From these we can calculate $r^\alpha,\ w^\alpha_p$ and $w^\beta_p$. What's more, if we look at the curve $\rho=1/k$ we may count the number of turning points of this curve, $\nu$; these correspond to vertices of $\ba$ which lead to cusps of $\bbet$. From another work of the first author \cite{TWclocconv} we can therefore find $r^\beta$ via $r^\beta=(\nu+2r^\alpha)/2$. In summary \[ \begin{pmatrix}
r^\alpha \\ r^\beta \\ w^\alpha_p \\ w^\beta_p
\end{pmatrix}=\begin{pmatrix}
1 & 0 & 0 & 0 \\ 
1 & \tfrac{1}{2} & 0 & 0 \\ 
1 & 0 & -1 & 0 \\ 
1 & 0 & 0 & -1  
\end{pmatrix}\begin{pmatrix}
i_N \\ \nu \\ m_p \\ n_p
\end{pmatrix}-\begin{pmatrix}
1 \\ 1 \\ 1 \\ 1
\end{pmatrix}, \] and all the elements of $(i_N, \nu, m_p, n_p)$ can be quickly seen by looking at the normal map as a vector field.

\section{Normal map for surfaces}

In this section we extend the dicussion to consider the normal map of surfaces in $\rre^3$. To fix ideas we at first suppose $\mS:\rS^2\to\rre^3$ is a regular, smooth, convex topological sphere, but later we consider the more general case of $\mS$ having self-intersections. If $\bsig=\bsig(u,v)$ is a coordinate patch then we fix the orientation as follows: $\bn$ is inward pointing so the principal curvatures are positive, and $u,v$ are ordered so $(\bsig_{,u},\bsig_{,v},\bn)$ is right handed; then $F(u,v,\rho)=\bsig+\rho\bn$ is the normal map as before. As $DF$ is now $3\times 3$ there are more possibilities: for a given $u,v$ there will be two values of $\rho$ where the normal map has corank 1: $\rho=1/k_1$ and $\rho=1/k_2$ where $k_1$ and $k_2$ are the principal curvatures at $u,v$ (we label $k_2\geq k_1>0$). The corank may be 2 at a point where $k_1=k_2$, known as an umbilic point . In this context the images under $F$ of the singular sets tend to be known as `focal sets' rather than evolutes (we will label them $\bbet_1$ and $\bbet_2$ respectively), and thus the surface $\mS$ will have two focal sets in $\rre^3$ which will intersect each other in the large but also at umbilic points. There is much in the literature to read: Porteous \cite{porteous} and Izumiya et al \cite{izumiya} classify the different types of focal point; Thielhelm \cite{hannes} provides in-depth visualization of these focal sets; Guti\'{e}rrez et al \cite{gutsoto} describe the neighbourhood of umbilic points and the necessity of at least two umbilic points on spheres is the Carath\'{e}odory conjecture (see Guilfoyle \cite{guilf}).  Interestingly the global picture of the focal sets is largely missing; an archetypal example to dwell on would be the focal sets associated with the triaxial ellipsoid, and even this picture is largely absent from the literature (but see Joets and Ribotta \cite{joetsrib} for a rare example). While the triaxial ellipsoid is illuminating, we alert the reader to the fact that even for convex spheres, such as the spherical harmonic surfaces described in \cite{TWspherical}, the focal sets are exceedingly complex and difficult to visualize in $\rre^3$; we think the vector field approach provides some clarity.

The normal map for $\mS$ is $\rS^2\times\rre\to\rre^3$ but vector fields in $\rre^3$ are very cluttered to visualize and so we will not provide a plot here; see a representative sketch instead in Figure \ref{3dpic}. As before there are critical points which represent normals to $\mS$ which pass through the point we have chosen to be the origin. We classify them as follows: suppose $(u^*,v^*,\rho^*)$ corresponds to a normal through the origin. Then $(u^*,v^*)$ is a stationary point of the function $D^2:\rS^2\to\rre^+:(u,v)\to\bsig.\bsig$, and a critical point of the corresponding gradient vector field on $\mS$, with Poincar\'{e} index \begin{align} \text{sgn}[(1-\rho k_1)(1-\rho k_2)]. \label{pind} \end{align} Similarly $(u^*,v^*,\rho^*)$ is a critical point of the normal map vector field and its Poincar\'{e} index is also \eqref{pind}. As $\mS$ is a sphere and $\chi(\rS^2)=2$ this means the sum of the Poincar\'{e} indices of the finite critical points of the normal map vector field is 2.

A Poincar\'{e} index of $+1$ or $-1$ in $\rre^3$ is somewhat less informative than in $\rre^2$, in that there are 8 non-equivalent cases (4 from each) to consider (rather than 5) \cite{wang}, however we can arrange the phase space as follows: there are two sheets $\rho=1/k_1$ and $\rho=1/k_2$ which touch at umbilic points. There are at least two critical points, with those above $\rho=1/k_1$ and below $\rho=1/k_2$ having Poincar\'{e} index $+1$ (maxima and minima of $D^2$) and those between the sheets having index $-1$ (saddles of $D^2$); see Figure \ref{3dpic} for the case of the triaxial ellipsoid. If we now perform a two-point compactification of $\rS^2\times\rre$ onto $\rS^3$, we again may consider the Poincar\'{e} index at the poles. Since $\chi(\rS^3)=0$ we know these indices must sum to $-2$; indeed direct calculations along the lines described after Theorem 2 shows the indices at the poles are both $-1$ (recall we are considering $\mS$ to be a convex sphere).

As we move the origin of the coordinate system (or move $\mS$ around in $\rre^3$) we may have the origin cross one of the focal sets; in the vector field this means a pair of critical points are created or annihilated via a saddle-centre bifurcation on either $\rho=1/k_1$ or $1/k_2$ depending on which focal set was crossed, but notice one critical point is a maximum or minimum of $D^2$ depending on whether we cross the $1/k_1$ or $1/k_2$ focal set respectively (the other critical point will be a saddle). Since the winding index of a focal set changes by $\pm 1$ when we cross it, and since when the point $p$ goes to infinity there will be precisely 2 normals through it, we can extend \eqref{wpbN} to $N_p=2(w^{\beta_1}_p+w^{\beta_2}_p)+2$, but this does not distinguish $w^{\beta_1}_p$ and $w^{\beta_2}_p$. Taking a hint from the comment above, instead of counting the number of normals through $p$ we count the number of normals which correspond to maxima, minima and saddles of $D^2$; let's label them $n^+_p,n^-_p$ and $n^0_p$ respectively. Then we have \[ w^{\beta_1}_p=n^+_p-1,\quad w^{\beta_2}_p=n^-_p-1, \label{b12} \] which is consistent with $n^++n^--n^0=2$.

More generally, $\mS$ may be the image of $\rS^2$ under some map that might have self-intersections - imagine the lima\c{c}on rotated about its axis of symmetry. In this context rather than the term `rotation index' the term `degree of the Gauss map' tends to be used \cite{milnor}; let's say the degree of the Gauss map of $\mS$ is $d$. The local structure at the critical points is as described previously, but now since the Euler characteristic is twice the degree of the Gauss map (see for example \cite{mikh}) the sum of the Poincar\'{e} indices of the finite critical points should be $2d$. When looking at the projection onto $\rS^3$ calculations suggest the Poincar\'{e} index at the poles is $-d$ and $-d$, which agrees with $\chi(\rS^3)=0$. Indeed we might conjecture an extension to \eqref{b12} as $w^{\beta_1}_p=n^+_p-d, w^{\beta_2}_p=n^-_p-d$. Rigorous proofs of the statements in this paragraph are more involved and might require an analogue of Lemma 1; as such we leave further consideration to future work.

\begin{figure}
\begin{center}
\includegraphics[width=\textwidth]{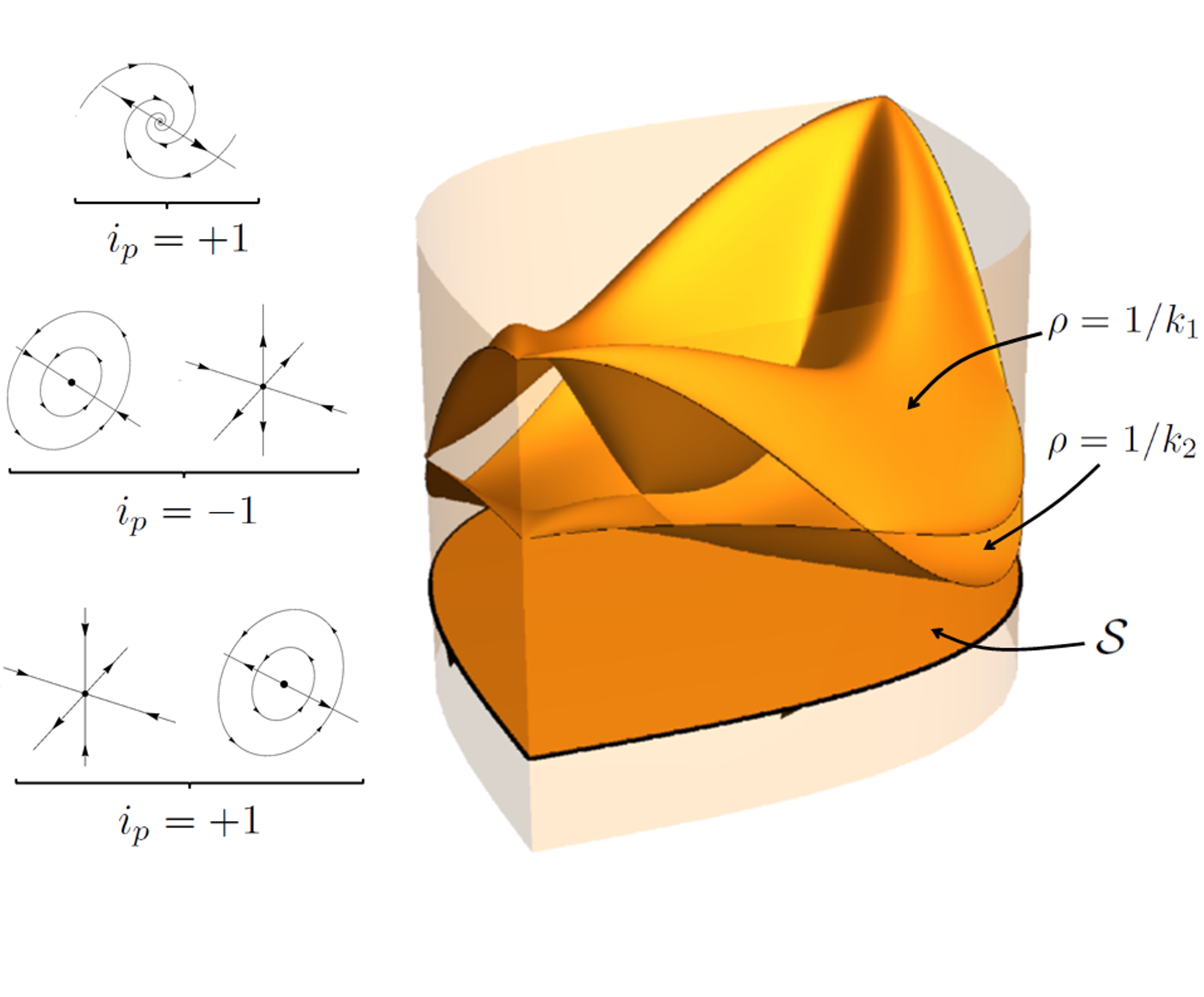}\caption{As described in the text, the two sheets $\rho=1/k_1$ and $\rho=1/k_2$ for the triaxial ellipsoid with semi-axes $(2,1.5,1)$. The surface $\mS$ has been opened into the lozenge shape by choosing two points as poles, cutting along the line joining those poles, and flattening the surface (hence the pointed end visible is one pole and the two curved adges are identified). The vector field for this triaxial ellipsiod, with origin of coordinates at the centre, would have two maximum/saddle/minimum critical points above/between/below the sheets; the linear structure for each case is shown (above $\rho=1/k_1$ both critical points have the same eigenvalues). }\label{3dpic}
\end{center}
\end{figure}

\section{Conclusions}

We have taken a novel approach in this paper by viewing the normal map as a vector field, and by studying the critical points of this vector field and making the connection with the distance function we have derived some simple counting formulae. In this paper we have only considered the case of positive curvature. If we allowed the curvature to change sign, then where $k$ passes through 0 the curve $1/k$ in Figure \ref{normspic} would diverge, and then when we map to the sphere this curve would lead to loops passing through the poles. The results from Sections 2 and 3 would need to be adapted as follows: Theorem 1 holds if we replace ``the saddles all lie above the line $\rho=1/k$ and the centres below it'' with ``the saddles lie above/below the line $\rho=1/k$ in the regions where $k$ is positive/ negative respectively, with the centres below/above it''. Theorem 2 holds with no change, however the problem with Theorem 3 is how to meaningfully define $w_p^\beta$ if $\bbet$ is unbounded and disconnected (one option would be to `complete' the evolute by adding the normal lines to $\ba$ where $k=0$, oriented by the asymptotic portions of the evolute). Theorem 4 would become \[ w_p^\alpha=r^\alpha-(m_p^+-m_p^-) \] where $m_p^+/m_p^-$ are the numbers of normal segments with $\rho<0$ through $p$ which correspond to minima/maxima of the distance function respectively.

For surfaces, even when we restrict our attention to convex surfaces, the focal sets can be very complex. An interesting question would be to consider the generalisation of the formula given at the end of Section 3, for the rotation index of the focal sets. Now instead of cusps there are ribs (\cite{porteous},\cite{izumiya}), and there may be a formula for the rotation index in terms of the number of ribs, or (in analogue with the four cusp theorem) a `three ribs' theorem. In summary, the caustic of normal lines to curves and surfaces has been much studied over the years but nonetheless still provides interesting and complex structure warranting further research.

\bibliographystyle{plain}

\bibliography{clbib}

\end{document}